\documentclass[a4paper,10pt]{article}
\usepackage[latin1]{inputenc}
\usepackage[T1]{fontenc}
\usepackage{amsmath}
\usepackage{amsfonts}
\usepackage{amstext}
\usepackage{amsthm}
\usepackage[all]{xy}
\usepackage{geometry}
\usepackage{graphicx}
\usepackage{slashed}
\newtheorem{theorem}{Theorem}[section]
\newtheorem{lemma}[theorem]{Lemma}
\newtheorem{defi}[theorem]{Definition}
\newtheorem{rem}[theorem]{Remark}
\newtheorem{prop}[theorem]{Proposition}
\newtheorem{cor}[theorem]{Corollary}

\DeclareMathOperator{\im}{im}
\DeclareMathOperator{\ind}{ind}

\DeclareMathOperator{\sfl}{sf}

\DeclareMathOperator{\sgn}{sgn}

\newcommand{\C}{\mathbb{C}}
\newcommand{\R}{\mathbb{R}}

\newcommand{\Z}{\mathbb{Z}}

\newcommand\cA{\mathcal{A}}

\newcommand\cD{\mathcal{D}}

\newcommand\cL{\mathcal{L}}
\newcommand\cM{\mathcal{M}}
\newcommand\cN{\mathcal{N}}

\newcommand\cR{\mathcal{R}}
\newcommand\cS{\mathcal{S}}

\newcommand\cU{\mathcal{U}}

\newcommand{\Id}{\mathrm{Id\,}}
\newcommand{\ve}{\varepsilon}
\newcommand{\ga}{\gamma}

\newcommand{\la}{\lambda}

\newcommand\fs{\Phi_{S}(H)}

\newcommand{\sig}{\operatorname{sig}}
\newcommand{\spf}{\operatorname{sf}}

\newcommand{\ip}[2]{\langle #1, #2\rangle}

\newcommand{\hand}{\hbox{ \,and\, }}
\newcommand{\hif}{\hbox{ \,if\,} }

\newcommand{\bt}{\begin{theorem}}
\newcommand{\et}{\end{theorem}}
\newcommand{\bc}{\begin{corollary}}
\newcommand{\ec}{\end{corollary}}
\newcommand{\bp}{\begin{proposition}}
\newcommand{\ep}{\end{proposition}}
\newcommand{\bl}{\begin{lemma}}
\newcommand{\el}{\end{lemma}}
\newcommand{\er}{\end{remark}}
\newcommand{\bd}{\begin{definition}}
\newcommand{\ed}{\end{definition}}
\newcommand{\be}{\begin{equation}}
\newcommand{\ee}{\end{equation}}
\newcommand{\ba}{\begin{array}}
\newcommand{\ea}{\end{array}}
\newcommand{\ra}{\rightarrow}

\title{Bifurcation of critical points for continuous families of $C^2$ functionals of Fredholm type}
\author{Jacobo Pejsachowicz and Nils Waterstraat}

\begin{document}
\date{}
\maketitle

\begin{abstract}
Given a continuous family  of $C^2$ functionals of Fredholm type,  we show that the non-vanishing of the spectral flow for the family of Hessians along a known (trivial) branch of critical points  not only entails bifurcation of nontrivial  critical points  but also  allows to estimate the number of bifurcation points along the branch.  We use this result for several parameter bifurcation, estimating  the number of connected components of the complement of the set of bifurcation points  and  apply our results  to bifurcation of  periodic orbits of Hamiltonian systems.  By means of a comparison principle  for the spectral flow,  we  obtain lower bounds for  the number of bifurcation points of periodic orbits on a given interval in terms of the coefficients of the linearization.
\end{abstract}

\maketitle

\footnotetext[2]{J. Pejsachowicz  was supported by  MIUR-PRIN 2009-Metodi variazionali e topologici nei fenomeni nonlineari.}
\footnotetext[3]{N. Waterstraat was supported by the German Academic Exchange Service (DAAD) and  GNAMPA-INdAM.}

\section{Introduction}
In this paper we deal with bifurcation of critical points of a continuous family of $C^2$ functionals $\{\psi_\lambda \colon U \rightarrow\mathbb{R}\}_{\lambda\in\Lambda}$, defined on an open neighborhood $U$ of $0$ in a separable Hilbert space $H$  and parametrized by a topological space $\Lambda$. Such a family is given by a function   $\psi\colon\Lambda\times U \rightarrow\mathbb{R}$ such that,  for any $\lambda \in \Lambda$, the map  $\psi_\lambda \equiv\psi(\lambda,\cdot)\colon U \rightarrow\mathbb{R}$ is $C^2$, and moreover for any $k$, $1\leq k\leq 2$, the map sending the point $(\lambda, x)$ to the $k$-th differential $d^k \psi_\lambda(x)$ is a continuous map from $\Lambda\times U$ into  the normed space of symmetric $k$ forms, $Sym^k(H,\mathbb{R}).$ In what follows, we will always assume that the point $0$ is a critical point of  $\psi_\lambda$, for all $\lambda\in \Lambda$. Families of functionals of this type arise when the data of the problem are not smooth enough in order to ensure the differentiability with respect to the parameter. \\
Let $f_\lambda(x)  = \nabla\psi_\lambda(x)$ be the gradient of $\psi_\lambda$ at the point $x\in U$, and let $L_\lambda=Df_\lambda(0)$  be the Fr\'{e}chet differential  of $f_\lambda$ at the critical point $0\in H$. The operator  $L_\lambda$ is the Hessian of $\psi_\lambda$ at the critical point $0$. Each $L_\lambda$ is a bounded self-adjoint operator and $\psi_\lambda(x)=\langle L_\lambda x, x\rangle +o(\|x\|^2).$\\
Since $\psi$ is  a continuous family of $C^2$ functionals, $f$ is continuous in both variables and differentiable in $x$. Moreover, the map $L\colon\Lambda\rightarrow \mathcal L_S(H)$ defined by $L(\lambda) =L_\lambda$  is continuous with respect to the norm topology in the space $\mathcal L_S(H)$ of all bounded self-adjoint operators from $H$ into itself.  We assume in addition that each $L_\lambda$ is Fredholm, namely, that it has a closed image and finite dimensional kernel. From the self-adjointness of $L_\lambda$, we have that $\ind L_\lambda = \dim \ker L_\lambda - \text{codim}\, \im L_\lambda=0$.\\
The space of all self-adjoint Fredholm operators $\Phi_S(H)$ is an open subset of $\mathcal L_S(H)$. Therefore, by considering possibly a smaller neighborhood $U$ of $0$, we can suppose that $Df_\lambda(x)$ is  Fredholm for every $x\in U.$ Thus $f$ is a continuous family of $C^1$-Fredholm maps possessing a variational structure,  with $\psi$ as associated family of potential functions. A  family $\psi$ as above will be called  {\it  a continuous family of Fredholm $C^2$ functionals.} The set $\mathcal T\equiv \Lambda\times \{0\}$ is  called the {\it trivial branch.}  
 
\begin{defi}
A point  $\lambda_\ast\in\Lambda$ is a point of {\sl bifurcation  of critical points} of the family $\psi$ from the trivial branch if  every neighborhood of $(\lambda_\ast,0)$ in $\Lambda \times U$ contains some point $(\lambda,x)$, where $x\neq 0$ is a critical point of $\psi_\lambda.$
\end{defi}

By the implicit function theorem,  bifurcation can occur only at points $\lambda\in\Lambda$ where $L_\lambda$ fails to be invertible. In this paper  we will  deal   first with the case  where  $\Lambda=[a,b]$ is a compact interval and then we will use  the obtained results  in several-parameter bifurcation.  We assume that both $L_a$ and $L_b$ are invertible and will look for sufficient conditions ensuring the existence of at least one bifurcation point for the family $\psi$ in $(a,b).$\\ 
When $\psi $ is $C^2$ and $H$ is finite dimensional  it is a folklore result that bifurcation arises in $(a,b)$ whenever the Morse index  $\mu(\psi_a,0)$ differs from $\mu(\psi_b,0)$. This result  extends to $C^2$ functionals $\psi\colon [a,b]\times U\ra \R$  if the Hessians  $L_\la $ of $\psi_\la$ at $0$ are either essentially positive or essentially negative. Let  us recall that  a self-adjoint operator $L_\la$ is  essentially positive  if it is   a compact perturbation of a positive definite  self-adjoint operator.  If $\psi_\la$ is as above and if  $ 0$ is a non-degenerate critical point of $\psi_\la, $  its  Morse index  $\mu(\psi_\la, 0)$  is the dimension of the maximal negative space of its Hessian $L_\la$ at $0$. With this definition the corresponding bifurcation theorem holds in the  form stated  in finite dimensions. The same holds for functionals with essentially negative definite Hessian.\\
The above results are sufficient for many applications  of bifurcation theory to nonlinear differential equations, but not for all of them. Indeed, in variational problems arising in the theory of Hamiltonian systems, e.g., geodesics on semi-Riemannian manifolds, and perturbations of self-adjoint systems of first order elliptic differential operators, one deals with strongly indefinite functionals $\psi$, where $\mu(\psi,x)$ and $\mu(-\psi,x)$ are both infinite. In this case the invariant that substitutes the difference between the Morse indices at the end points is the {\sl spectral flow} of a path of self-adjoint Fredholm operators.\\
Roughly speaking, given a path $L\colon I \rightarrow \Phi_S(H)$ of self-adjoint Fredholm operators with invertible end-points, its spectral flow $\sfl(L,I)\in\mathbb{Z}$ is the number of negative eigenvalues of $L_a$ that become positive as the parameter $\lambda$ travels from $a$ to $b$ minus the number of positive eigenvalues of $L_a$ that become negative. When the operators of the path are essentially positive, then the spectral flow of $L$ is the difference of the Morse indices at the end-points. However, if the operators in the path are strongly indefinite, i.e. neither $L$ nor $-L$ consists of essentially positive operators, then $\sfl(L,I)$ depends in general on the whole path and not only on its end-points.\\ 
The spectral flow was  introduced for the first time for paths of elliptic  self-adjoint operators in \cite{AtPatSi76} and since then its definition was extended with  various degrees of generality and used in  linear and nonlinear functional analysis. Let us quote as references \cite{BosandWoj85, Floer88, CapLeMi, RS, Phillips, MPP, Wah, PW}.\\
The above heuristic description of  the  spectral flow  can be made rigorous in many different ways.  A neat geometric understanding  of the spectral flow can be obtained  by interpreting this invariant as  an oriented intersection number of the path with the singular  variety  $\Sigma$  of all non-invertible  self-adjoint Fredholm operators. The  set   $\Sigma$  is a one codimensional stratified subvariety of the open subset $\Phi_S(H)\subset\mathcal{L}_S(H)$ of all bounded self-adjoint Fredholm operators on $H.$   For $k\geq 1,$  the  stratum  $\Sigma_k =\{ T\in \Phi_S(H): \dim  \ker T=k \}$ of the variety  $\Sigma $ is a submanifold (not closed in $\Phi_S(H)$). The  top stratum $\Sigma _1$ is of codimension  $1$  in $\Phi_S(H),$ while the other strata are of codimension three or higher. $\Sigma_1$ posses a nowhere vanishing normal vector field, i.e., is co-oriented.  This allows us to define  $ \sfl(L,I)$ by approximating the path $L$ with a smooth path having transversal intersections with $\Sigma$ (hence, not intersecting the  strata of higher order)  and  counting the intersection points of the  approximating path with  $ \Sigma _1$   with  signs $\pm 1$ according to whether the orientation of the tangent vector to the path coincides or not with that of a chosen normal field (cf. \cite{FPR}).\\
The graph of any  self-adjoint operator is a  Lagrangian subspace of the product $H\times H$ with its natural (cotangent) symplectic structure. In this picture $\Sigma$  becomes the  "train" (or "Maslov Cycle") of $H\times\{0\}.$  Using this, the spectral flow can be interpreted as an infinite dimensional Maslov index for paths in the Fredholm Lagrangian Grassmannian. This approach leads to the definition of spectral flow for families  of unbounded self-adjoint operators as well (cf. \cite{NicolaescuCRAS}).\\
That the spectral flow  is the right homotopy  invariant of the path of linearizations for the study of bifurcation of critical points of families of strongly-indefinite functionals was found in \cite{FPR}.  In the present paper we  will improve the  results of \cite{FPR} in several ways.\\
 First of all, we extend  the bifurcation theorem of  \cite{FPR}  to continuous families of $C^2$ functionals parametrized by an interval by showing  that bifurcation of critical points arise  whenever  the path of Hessians along the trivial branch has a non-vanishing spectral flow.  For this we  will  use the homotopy invariance of the Conley index of an isolated invariant set in the central part of the proof of the theorem, after having reduced the problem to finite dimensions. The Conley index needs less differentiability assumptions about the functional  than the homology groups of a critical point  used in \cite{FPR}.  In addition  we  estimate the number of   bifurcation  points of the family $\psi$  in terms of the spectral flow $\sfl (L,I) $ of the path of Hessians  along the trivial branch and the highest  stratum $\Sigma_m $ crossed by $L.$   Our  conclusion in Theorem \ref{thm:bif} is as follows: if  the crossing points with $\Sigma$  are isolated,  then  $\psi$ must have at least $|\sfl (L,I) |/m$ bifurcation points.\\
When the parameter space is a general topological space $\Lambda,$ upon appropriate assumptions, we will use  the above result along  paths in  $\Lambda$  in order to  estimate   the number of connected components  of the complement of the set of bifurcation points. This is done in Theorem \ref{thm:fam}.\\
Bifurcation of periodic and homoclinic orbits of Hamiltonian systems is a natural place for applications of the above theory since the associated functionals are of strongly indefinite type.  We apply our results to  bifurcation of $2\pi$-periodic orbits of  time-depending, periodic  Hamiltonian systems whose Hamiltonian  function has  a time independent Hessian. In this case the linearized equation has constant coefficients and from this fact  one easily obtains sufficient conditions for bifurcation and estimates for the number of components  directly in terms of the coefficients.\\
The case of  general non-autonomous periodic  Hamiltonian systems  is more involved. However, for the one-parameter case, we obtain in Proposition \ref{prop:last} some  estimates for the number of bifurcation points for $2\pi$-periodic orbits in terms of the coefficients of its  linearization along the stationary branch.  The estimate is deduced from a general  comparison principle  for the spectral flow which is stated in Theorem \ref{thm:comparison} and might have other uses as well.\\
The paper is organized as follows: in Section \ref{sec:mt} we state our main results.  In Section \ref{sec:sfl} we recall some well known properties of the spectral flow and prove Proposition \ref{prop:fingen}. The sections  \ref{sec:bif} and \ref{sec:fam} are devoted to the proofs of the theorems \ref{thm:bif} and \ref{thm:fam}, respectively. Applications to several parameter  bifurcation of periodic orbits of Hamiltonian systems with time independent Hessians are given in Section \ref{sec:apl}.  In  Section \ref{sec:comp}  we establish the  comparison principle which is used in Section \ref{sec:est} in order to obtain an estimate for the number of bifurcation points for non-trivial periodic orbits of a Hamiltonian system from the stationary branch.

%%%%%%%%%%%%%%%%%%%%%%%%%%%%%%%%%%%%%%%%%%%%%%%%%%%%%%%%%%%%%%%%%%%%%%%%%%%%%%%%%%%%%%%%%%%%%%%%%%%%%%%%%%%%%%%%%%%%%%%%%%%%%%%%%%%%%%%%%%%%%%%%%%%%%%%%%%%%%%%%%%%%%%%%%%%%%%%%%%%%%%%%%%%%%%%%%%%%%%%%%%%%%%%%%%%%%%%%%%%%%%%%%%%%%%%%%%%%%%%%%%%%%%%%%%%%%%%%%%%%%%%%%%%%%%%%%%%%%%%%%%%%%%%%%%%%%%%%%%%%%%%%%%%%%%%%%%%%%%%%%%%%%%%%%%%%%%%%%%%%%%%%%%%%%%%%%%%%%%%%%%

\section{The main theorems}\label{sec:mt}
  
Our main result reads as follows: 

\begin{theorem}\label{thm:bif}
 Let  $U$ be a  neighborhood of  $0$ in a separable Hilbert space $H$ and let  $\psi\colon I \times U \rightarrow\mathbb{R}$ be a  continuous family  of  $C^2$  functionals parametrized by $I =[a,b].$  Assume that $0$ is a critical point of the functional $\psi _\lambda \equiv \psi(\lambda,\cdot)$ for each $\lambda\in I$. Moreover, assume that the Hessians  $L_\lambda$  of $ \psi_\lambda $ at $0$ are Fredholm  with   $L_a$  and  $L_b$ invertible.
 
 \begin{itemize}
  \item[i)] If $\sfl(L,I)\neq 0$, then the interval $(a,b)$ contains at least one point of bifurcation of critical points of $\psi_\lambda$ from the trivial branch.\\
  
\item[ii)]  If $L$ intersects $\Sigma$ only at a finite number of points  $\lambda\in I$,  then the family $\psi$ possesses at least 
$|\sfl(L,I)|/m$
bifurcation  points in $(a,b),$ where $m$ is the highest  order of the stratum crossed by $L,$ i.e., 
 \[m=\text{max}\{ \dim \ker L_\lambda: \lambda \in [a,b]\} .\]  
\end{itemize}
\end{theorem}

\vspace{0.5cm}

There is a number of cases in which the hypothesis of ii)  are verified: 
\begin{itemize} 
\item[1)]  The path $L$ is real analytic. Indeed,  in this case,  the set  $ \Sigma(L):= L^{-1}(\Sigma)$ of all  {\it singular points}  of the path $L$ has to be discrete, because   on a small neighborhood of any point  $\la\in I$  the set $\Sigma(L)$ coincides with the set of zeroes of the analytic function  $\det(U^*LU_{| H_0}),$ where $U$ and $H_0$ are as in the proof of Lemma \ref{sflred} below. 

\item[2)] The path $L$ is differentiable and has only regular crossing points in $I.$  A regular crossing point  is a point  $\lambda \in \Sigma(L)$ at which the crossing form $\mathcal{Q}(\lambda),$ defined as the restriction of the quadratic form $\ip{\dot{L}_\lambda h}{h}$ to  $\ker L_\lambda,$  is non-degenerate.  The proof  that regular crossing  points are isolated  can be found in  \cite[Theorem 4.1]{FPR}. It also follows from the above  theorem that, for paths having only regular crossing points, one has 
\begin{equation}\label{sumsfl}
\sfl(L,I) = \sum_{ \la\in \Sigma(L)} \sig \mathcal{Q}(\la),
\end{equation}
where $\sig$ stands for the signature of a quadratic form.
It is easy to  see that at regular crossing points  $ \sig \mathcal{Q}(\la)$ coincides with the  "crossing number"  studied in \cite{CL, K}. From this viewpoint the spectral flow  arise  in variational  bifurcation theory as an improvement of this, earlier defined invariant at isolated points in $\Sigma(L).$

\item[3)] A differentiable path is  said to be  positive if the quadratic form $\mathcal{Q}(\la)$ is positive definite at each  crossing point. By the above discussion, positive paths have only regular crossings and moreover  $\sig \mathcal{Q}(\la)= \dim \ker L_\la.$ 
Therefore positive paths verify the second hypothesis in Theorem \ref{thm:bif}. Moreover, in this case  \[\sfl(L,I)= \sum_{\la \in \Sigma(L)}\dim \ker L_{\la}.\] 

A typical positive path on $(0,\infty)$ is $L_\la = \la  \Id - K$, where $K $ is a compact operator. A well known theorem by Krasnoselskii states that  if  $\phi$ is  a weakly continuous functional  such that $\nabla\phi(0)=0,$  then  every  non-vanishing eigenvalue of the Hessian of $\phi$ at $0$ is a bifurcation point  for solutions  of the variational equation  $ \la x- \nabla\phi( x)=0.$  Thus  Krasnoselskii's  theorem is  a very special case of Theorem \ref{thm:bif}.

\item[4)]  The  positivity  can be   formulated for paths that are only continuous,  by requiring that for each singular point $\la_0$ of $L$ in $I$ there is a neighborhood $ I_\delta =(\la_0-\delta, \la_0+\delta)$ and  an increasing function $\gamma \colon I_\delta \rightarrow \R$ with  $\gamma(\la_0)=0,$ such that:
\begin{align*}
\ip{L(\la)-L(\la_0)u}{u}&\le \gamma(\la)\ip{u}{u} \quad\forall\, \la \in (\la_0-\delta , \la_0],\,\, \forall \, u\in \ker L_{\la_0},\\
\ip{L(\la)-L(\la_0)u}{u}&\ge \gamma(\la)\ip{u}{u}\quad\forall\, \la \in [\la_0, \la_0 +\delta),\,\, \forall  \, u\in \ker L_{\la_0},
\end{align*}
and 
$\lim_{\la\to 0}\frac{\Vert L(\la)-L(\la_0)\Vert^2}{\gamma(\la)}=0.$
In this case (cf. \cite{FPS}), one still has that the singular points are  isolated and 
$\sfl(L,I)= \sum_{\la \in \Sigma(L)}\dim \ker L_{\la}.$
\end{itemize}

\vskip10pt

Using the above theorem, we will  obtain estimates on the number of connected components of the complement of the set of bifurcation points of families of functionals parametrized by more general topological spaces. More precisely, we assume that $\Lambda$ is a connected topological space and $\psi$ a continuous family of Fredholm $C^2$ functionals such that $0\in H$ is a critical point of all $\psi_\lambda$, $\lambda\in\Lambda$.\\
 Let us denote the set of all bifurcation points in $\Lambda$ by $B(\psi)$. In what follows, we call a path $\gamma:I\rightarrow\Lambda$ admissible if the operators $L_{\gamma(a)}$ and $L_{\gamma(b)}$ are isomorphisms. The spectral flow $\sfl(L\circ{\gamma},I)$ of $L$ along an admissible path is well defined and it is additive under concatenation. We will say that  the family  $\psi$ satisfies the assumption (A) if  for any admissible path $\ga$ the spectral flow $\sfl(L\circ{\gamma},I)$ depends only on the end-points of the path, or equivalently,  if $\sfl(L\circ{\gamma},I)=0$ for all closed admissible paths. 

\begin{theorem}\label{thm:fam}
Let $\Lambda$ be a connected topological space and $\psi:\Lambda\times H\rightarrow\mathbb{R}$ a continuous family of Fredholm $C^2$ functionals such that $0\in H$ is a critical point of all $\psi_\lambda$, $\lambda\in\Lambda$, and such that (A) holds.

\begin{itemize}
	\item[i)] If there exists an admissible path $\gamma$ in $\Lambda$ such that $\sfl(L\circ\gamma,I)\neq 0$, then $\Lambda\setminus B(\psi)$ is disconnected.
	
	\item[ii)] If there exists a sequence of admissible paths $\gamma_n$, $n\in\mathbb{N}$, such that
	$\lim_{n\rightarrow\infty}|\sfl(L\circ\gamma_n,I)|=\infty,$ then $\Lambda\setminus B(\psi)$ has infinitely many path components. 

 \item[iii)]   If $ \Sigma(L) = B(\psi),$ any admissible path $\ga$ such that  $L\circ\ga$  has only isolated singular points will  cross at least   $\frac{|\sfl(L\circ\gamma)|}{m} +1 $ components of $\Lambda \setminus B(\psi),$ where $m$ is defined as in Theorem \ref{thm:bif}. 
\end{itemize}
\end{theorem}

\begin{rem}
{\rm Let us recall that the Lebesgue covering dimension $\dim\Lambda$ of a topological space $\Lambda$ is the minimal value of $n\in\mathbb{N}$ such that every finite open cover of $\Lambda$ has a finite open refinement in which no point  belongs to more than $n+1$ elements. By Corollary 1, Theorem IV 4 of \cite {Hurewicz},  no subset of dimension strictly smaller than $n-1$ can disconnect a topological $n$-manifold. Therefore, it follows from Theorem \ref{thm:bif} that,  if the parameter space $\Lambda$ is a topological manifold of dimension $n,$ then the covering dimension of $B(\psi)$ is at least $n-1.$} \end{rem}
 
\begin{cor}\label{cor:fam} If $\Lambda$ is  a smooth manifold verifying the assumptions of  Theorem \ref{thm:fam} and  if $\Sigma(L)\subset \cS,$  where $\cS$ is a stratified submanifold  of  $\Lambda$ of  positive codimension,  e.g., if $L$ is smooth and transversal to the variety $\Sigma,$  then iii) holds irrespective of the assumption $ \Sigma(L) = B(\psi).$  The same is true if $ \Lambda $ is a topological  $n$-manifold  and  $\dim \Sigma (L) \setminus B(\psi) \leq n-2.$  \end {cor}

The condition  (A) is satisfied if the Hessians  $L_\lambda$ are either essentially positive or negative definite.  However, we will use (A) in the case of a family of Hamiltonian systems, in which none of the Morse indices is finite. In this framework (A)  still holds by \eqref{murel} below.\\
The following proposition shows the relevance of the topology of the parameter space for the validity of (A).

\begin{prop}\label{prop:fingen}
Let $\Lambda$ be a connected space of finite type, i.e., all homology groups $H_k (\Lambda)$, $k\in\mathbb{N}$, are finitely generated. If $H_1(\Lambda;\mathbb{Q})=0$, then (A) holds for any continuous family of Fredholm $C^2$ functionals parametrized by $\Lambda$. 
\end{prop}
Finally, let us mention that Theorem \ref{thm:fam} and Proposition \ref{prop:fingen} improve the main results of the recent work \cite{PW}. Moreover, the proof of the second assertion in Theorem \ref{thm:fam} simplifies the argument used in \cite{spinors}.

%%%%%%%%%%%%%%%%%%%%%%%%%%%%%%%%%%%%%%%%%%%%%%%%%%%%%%%%%%%%%%%%%%%%%%%%%%%%%%%%%%%%%%%%%%%%%%%%%%%%%%%%%%%%%%%%%%%%%%%%%%%%%%%%%%%%%%%%%%%%%%%%%%%%%%%%%%%%%%%%%%%%%%%%%%%%%%%%%%%%%%%%%%%%%%%%%%%%%%%%%%%%%%%%%%%%%%%%%%%%%%%%%%%%%%%%%%%%%%%%%%%%%%%%%%%%%%%%%%%%%%%%%%%%%%%%%%%%%%%%%%%%%%%%%%%%%%%%%%%%%%%%%%%%%%%%%%%%%%%%%%%%%%%%%%%%%%%%%%%%%%%%%%%%%%%%%%%%%%%%%%

\section{The spectral flow} \label{sec:sfl}

The proof  of Theorem \ref{thm:bif} is based  on a  reduction process  to families of  functionals of a particular form  and some properties of the spectral flow which we are going to review in this section. Among  the  several constructions  of this invariant in the literature, we will follow the approach taken in \cite {FPR}  because it  leads almost immediately to the  reduction process mentioned above.\\
Let us recall that the open set  $\Phi_S(H)$ is not connected and that the position  with respect to $0$ of the essential spectrum of an operator determines to which connected component the operator belongs. The three connected components of $\Phi_S(H)$ are: the set of \textit{essentially positive operators} $\Phi^+_S(H)$ consisting of all operators in $\Phi_S(H)$ whose negative spectrum has only isolated eigenvalues of finite multiplicity; the set of \textit{essentially negative operators} $\Phi^-_S(H)$ defined in a corresponding way and the set of \textit{strongly indefinite operators} $\Phi^i_S(H)$ whose elements have infinite dimensional negative and positive spectral subspaces.\\
Let us  briefly recall from \cite{FPR} the definition of the spectral flow of  a path $L:I\rightarrow\Phi_S(H)$. In accordance with the notation in Theorem \ref{thm:fam}, $L$ will be called {\sl admissible} if its end-points are invertible. At first, if such an admissible path $L$ is a compact perturbation of a fixed self-adjoint Fredholm operator $T\in\Phi_S(H)$, i.e.,  $L_\lambda = T + K_\lambda$  where $K_\lambda$ is compact,  then the spectral flow is defined as  the relative Morse index of the end-points of the path; namely:

\begin{equation}  \label{murel}
\begin{aligned}
\sfl(L, [a,b])&=\mu_{rel}(L_a,L_b)\\&=\dim [ E^-(L_a)\cap E^+(L_b)]- \dim[E^-(L_b)\cap E^+(L_a)],
\end{aligned}\end{equation}
where $E^\pm$ denote the maximal positive ( resp.  negative)  space of a self-adjoint operator. In order to explain how \eqref{murel} can be extended to general admissible paths in $\Phi_S(H)$, we firstly assume that $L:I\rightarrow\Phi^i_S(H)$ is a path of strongly indefinite operators. We split $H$ into a product $H^+\times H^-$ with both $H^\pm$ isomorphic to $H$ and, writing $h=(x,y),$ we define $J(x,y) = x-y$. Clearly, $J^2=Id $ and the  spectrum of $J$ is  $\{ \pm 1\}$, both with infinite dimensional spectral subspace. Such a $J$ is called polarization in \cite{W}. Once a polarization is chosen, we say that a path $M\colon I\rightarrow GL(H)$ is a cogredient parametrix for a path  $L\colon I \rightarrow\Phi^i_S(H)$ with respect to $J,$ if 

\begin{align}\label{eq:parametrix} 
M_\lambda^\ast L_\lambda M_\lambda = J + K_\lambda,
\end{align}
with $K_\lambda$ compact and (necessarily) self-adjoint. It is shown in \cite{FPR} that every path in $\Phi^i_S(H)$ has a cogredient parametrix $M$ and that the definition
\begin{align}\label{sfl}
\sfl(L, [a,b])=\mu_{rel}(J+K_a,J+K_b)
\end{align}
does not depend on the choice of the cogredient  parametrix $M$.\\
The case of  paths $L:I\rightarrow\Phi^\pm_S(H)$ of essentially positive or essentially negative operators reduces to the one  in $\Phi^i_S(H\oplus H)$  by  taking direct sum  $\mp \Id $ on the second factor. Consequently, \eqref{sfl} extends to all admissible paths in $\Phi_S(H)$. It follows easily from \eqref{sfl}  that for paths in $\Phi^\pm_S(H)$ the spectral flow coincides with the difference of the Morse indices of $\pm L $ at the end points.\\  
The spectral flow is uniquely characterized by the following four properties (cf. \cite{CFP}): 

\begin{itemize}
\item[i)]{\it Normalization:} If  $L$ is a path of isomorphisms, then 
$\sfl(L, I)= 0.$\\
\item[ii)]{\it  Morse Index:}   If $H$ is finite dimensional,  then $\sfl(L, I)= \mu(L_a)-\mu(L_b).$\\ 
\item[iii)]{\it  Direct Sum Property:}  If $L_1$ and $L_2$ are 
admissible paths on Hilbert  
spaces $H_1$  and $H_2$ respectively, then 
$ \sfl(L_1\oplus L_2, I)=\sfl(L_1, I)+ \sfl(L_2, I).$ \\
\item[iv)]{\it Homotopy Invariance Property:}      Let 
$H\colon [0,1]\times  I \rightarrow\Phi_S(H)$  be a family such that for 
each  $s\in [0,1],$ the path  $ H_s\equiv H(s,\cdot)$ is admissible. 
Then $\sfl(H_s, I)$ is independent of $s$.  
\end{itemize}

From the above four properties it follows also that the spectral flow is additive under the concatenation of intervals, and  invariant under free homotopies of closed paths. Namely (cf.\cite{FPR}):

\begin{itemize} 
\item[v)] {\it Concatenation:}  If the path $L\colon  I  \rightarrow\Phi_S(H)$
is admissible  both on $[a,c]$ and on $[c,b],$ then  
$ \sfl(L, [a,b])=\sfl(L, [a,c])+ \sfl(L, [c,b]).$\\
\item[vi)] {\it  Homotopy of Closed Paths:}
Let  the family $H\colon [0,1]\times  I\to \fs$ be  such that $H(\cdot, a)= H(\cdot, b)$
and   $  H(0,a), H(1,a)$ are invertible. 
Then, irrespective of the invertibility of the operators $H(\cdot,a),$
\[\spf (H_0, I) = \spf(H_1,I).\]
\end{itemize}

\vspace{0.7cm}

We will now use the last two properties of the spectral flow in order to prove Proposition \ref{prop:fingen}.\\
Let $ \pi_1(\Lambda, \la)$  be the fundamental group of $\Lambda$ with base point $\la.$  Since the spectral flow is homotopy invariant and additive under concatenation of paths, it induces a homomorphism $\bar\sfl \colon \pi_1(\Lambda, \la) \ra \Z$ which necessarily sends the commutator subgroup $[\pi_1,\pi_1]$ to $0$. Therefore $\bar\sfl $  factors through the quotient $q \colon \pi_1(\Lambda, \la) \ra \pi_1(\Lambda, \la)/[\pi_1,\pi_1].$ Under the identification $ \pi_1(\Lambda, \la)/[\pi_1,\pi_1]\simeq H_1(\Lambda; \Z)$, we obtain a homomorphism  $ h\colon H_1(\Lambda; \Z)\ra \Z$ such that, denoting with $[\ga] $ the homotopy class of  a closed path $\ga$ based at $\la, $  we have that  $\sfl(L\circ\ga)=h\circ q([\ga]).$\\
Since $\Lambda$ is of finite type, there exists $d\in\mathbb{N}\cup\{0\}$ and prime numbers $p_1,\ldots,p_N$ such that $H_1(\Lambda;\mathbb{Z})\cong\mathbb{Z}^d\oplus\mathbb{Z}_{p_1}\oplus\cdots\oplus\mathbb{Z}_{p_N}$. It follows from the universal coefficient theorem that $d=0,$ because $H_1(\Lambda;\mathbb{Q})=0$ by assumption. We conclude that every element of $ H_1(\Lambda;\mathbb{Z})$ is of finite order which, on its turn, implies that $h\equiv 0.$  Thus $\sfl(L\circ\gamma)=0,$ for any closed path based at $\la$  and hence, by  vi), for any closed admissible path.  This yields the  desired conclusion.

%%%%%%%%%%%%%%%%%%%%%%%%%%%%%%%%%%%%%%%%%%%%%%%%%%%%%%%%%%%%%%%%%%%%%%%%%%%%%%%%%%%%%%%%%%%%%%%%%%%%%%%%%%%%%%%%%%%%%%%%%%%%%%%%%%%%%%%%%%%%%%%%%%%%%%%%%%%%%%%%%%%%%%%%%%%%%%%%%%%%%%%%%%%%%%%%%%%%%%%%%%%%%%%%%%%%%%%%%%%%%%%%%%%%%%%%%%%%%%%%%%%%%%%%%%%%%%%%%%%%%%%%%%%%%%%%%%%%%%%%%%%%%%%%%%%%%%%%%%%%%%%%%%%%%%%%%%%%%%%%%%%%%%%%%%%%%%%%%%%%%%%%%%%%%%%%%%%%%%%%%%

\section{Proof of Theorem \ref{thm:bif} }\label{sec:bif}
First of all, let us notice that it is enough to prove the theorem in the  strongly indefinite case. Indeed, given any family of functionals  $\psi\colon I \times  U \rightarrow\mathbb{R}$ of Fredholm type, one can consider  the auxiliary family $\hat{\psi} \colon I \times U\times U\times U\rightarrow\mathbb{R}$ defined  by

\[\hat{\psi}(\lambda, w, u, v)\equiv \psi(\lambda, u) +
 \frac{1}{2}\| w\| ^2 -\frac{1}{2}\| v\| ^2.\]
Clearly, the bifurcation points of $\psi$ are the same as those of $\hat{\psi}$. Moreover, by the direct sum property,  the spectral flows of the Hessians along the trivial branch of critical points of $\psi$ and $\hat{\psi}$ are the same. But the Hessians of $\hat{\psi}$ are strongly indefinite.
Henceforth  we assume that $\psi$ is strongly indefinite and do a further reduction of the problem.\\ 
We choose a cogredient parametrix $M$ for the path $L$ of Hessians of $\psi$ and set \[\tilde\psi(\lambda,h)\equiv \psi\left(\lambda,M_\lambda (h)\right), \  \ \tilde f(\lambda,h) = M_\lambda^*f(\lambda,M_\lambda  (h)), \ \  \tilde L_\lambda =M_\lambda^*L_\lambda  M_\lambda.\]
Then  we have for all  $(\lambda,h)\in I\times H$
\begin{equation}
\nabla_h\tilde {\psi}(\lambda,h)=\tilde f(\lambda,h)\,
\quad \text{and}\quad
 D_h\tilde f(\lambda,h)=\tilde L _\lambda h.
\end{equation}   
Since $\{M_\lambda\}_{\lambda\in I}$ and $\{M_\lambda^\ast\}_{\lambda\in I}$ are paths of invertible operators, it is clear that nontrivial solutions of the equation $f(\lambda,h)=0$  correspond via $M_\lambda$  with nontrivial solutions of  $\tilde f(\lambda,h)=0$. Therefore, the bifurcations of critical points of  $\psi$ and $\tilde\psi$ arise at the same values of the parameter $\lambda$. Also the crossing points of $L$ and $\tilde L$ with each stratum $\Sigma_k$ are the same.  Moreover, it follows from the homotopy invariance property iv) that the spectral flows of $L$ and $\tilde L= M^\ast LM$ coincide. Indeed, the homotopy  $H(\lambda,t)= M^\ast_{t\lambda} L_\lambda M_{t\lambda}$ shows that $ \sfl( \tilde L, I) =\sfl(M_a^\ast LM_a)$. Since $GL(H)$ is connected, there exists a path of invertible operators from $M_a$ to $\Id$ and we obtain that

 \[ \sfl( \tilde L, I)=\sfl(M_a^\ast LM_a)=\sfl(L,I).\]
Consequently, by possibly replacing $\psi\colon I\times U\rightarrow\mathbb{R}$ with $\bar\psi\colon I\times U\rightarrow\mathbb{R}$, we may  assume that for each parameter $\lambda\in I$, $L_\lambda=J+K_\lambda $ with $K_\lambda$ compact and symmetric, which we will do from now on.\\
We choose Hilbert bases $\{e_k^+\}_{k=1\dots  \infty}$ and $\{e_k^- \}_{k=1\dots\infty}$ for $H^+$ and $H^-$, respectively. Let $H_n$ be the subspace of $H$ spanned by $\{e_k^{\pm}|1\leq k\leq 
n\}.$ Then $J $ commutes with the orthogonal projection $P_n$ of $H$ onto  $H_n$ and hence $J(H_n)=H_n$, $J(H_n^\perp)=H_n^\perp.$  Namely, the pair $H_n, H_n^\perp$  reduces $J$. Notice also that the signature of the restriction  of $J$  to $H_n$ is zero.\\
It is shown in \cite[Lemma 4]{CFP} that there exists $n\in\mathbb{N}$ such that, for all $\lambda\in I,$

\[(\Id-P_n)L_\lambda:H^\perp_n\rightarrow H^\perp_n\]
is invertible, and for $0\leq t\leq 1$ and $\lambda\in\{a,b\}$,

\[tL_\lambda+(1-t)[(\Id-P_n)L_\lambda(\Id-P_n)+P_nL_\lambda P_n]\]
is invertible, too. Hence, if we denote 

\[L^n_\lambda=P_nL_\lambda\mid_{H_n}:H_n\rightarrow H_n\quad \text{and}\quad\tilde{L}^n_\lambda=(\Id-P_n)L_\lambda\mid_{H^\perp_n}H^\perp_n\rightarrow H^\perp_n,\] 
we get from iv), iii) and i)

\[\sfl(L,I)=\sfl(L^n\oplus\tilde{L}^n,I)= \sfl (L^n,I)+ \sfl(\tilde L^n, I)=\sfl(L^n,I).\]
Finally, by using ii), we obtain:

\begin{prop}\label{sflowfin}
If $L\colon I\rightarrow \mathcal L (H)$ and $L^n\colon I \to \mathcal L (H_n)$ are as above, then there exists $n_0$ such that 
\[\sfl(L,I) = \sfl(L^n, I) =  \mu(L^n_a) - \mu(L^n_b), \   \text{for all}\ n\geq n_0.\]
\end{prop}

Before turning to the proof of Theorem \ref{thm:bif} let us prove a version of the Lyapunov-Schmidt Reduction for critical points that we will use below. This later is a modification of the reduction proved  in \cite{FPR} adapted to the class of functionals we are working with. 

\begin{lemma}\label{lsreduction}  
Let $\psi\colon I \times  U\rightarrow\mathbb{R}$ be a continuous one-parameter family of $C^2$ functionals. Let $f(\lambda,h)\equiv\nabla_h\psi(\lambda,h).$ 
Assume that  $f(\lambda,0)=0$ for all $\lambda\in I.$ Suppose  that  there is an orthogonal  splitting $ H=X\times Y $, where $\dim X <\infty $ and such that,  writing  $h=(x,y) $ and  $f(\lambda,h)= ( f_1(\lambda,x,y), f_2(\lambda,x,y)),$  we have that  $D_y f_2(\lambda,0,0)\colon Y\rightarrow Y$  is invertible  for all $\lambda\in I.$ Then:
\begin{itemize}
\item[i)]  There is an open ball $B= B(0, \delta)\subset X$ and a continuous family of $C^1$ maps  
$g\colon I\times B\rightarrow Y$, such that $g(\lambda,0)=0$ for all $\lambda\in 
I,$ and  
\begin{align}\label{eq1}
f_2(\lambda, x, g(\lambda,x) )=0\,\,\text{ for all  }\,\,(\lambda,x)\in I\times B.
\end{align}
\item[ii)] If  the mapping $\bar f\colon I\times B\rightarrow X$ and the functional $\bar \psi\colon I\times B\rightarrow\mathbb{R}$ are  defined by \[\bar f(\lambda,x)\equiv f_1\left(\lambda,x, g(\lambda,x)\right) \ \hand \ 
\bar \psi(\lambda,x)\equiv\psi\left(\lambda,x, g(\lambda,x)\right)\ \text{respectively,} \]   
then $\bar\psi$ is a continuous family of $C^2$ functionals  on $B$ and 
\begin{equation}
\label{grad}
\nabla_x\bar \psi(\lambda,x)=\bar f(\lambda,x)\,\,\text{ for all }\,\, 
(\lambda,x) \in  I\times B. 
\end{equation}
\end{itemize}
\end{lemma} 
 
\begin{proof}
Composing  $f_2$ on the right with $D_y^{-1} f_2(\lambda,0,0)$, we obtain a map  $ \bar f_2 \colon I\times X\times Y \rightarrow Y$ such that $ D_y\bar f_2(\lambda,0,0)=\Id$ for all $\lambda\in I$. Clearly, there exists a $g$ verifying \eqref{eq1} for $f_2$ if and only if the same holds for $\bar f_2.$ Therefore, we can assume that 
 $D_y f_2(\lambda,0,0)=\Id$. Assuming this,  the derivative in the $y$ direction of the map  $k(\lambda,x,y)= y-f_2(\lambda, x,y)$ vanishes at $(\lambda,0,0)$ for all $\lambda$. Being  $D_y \,k(\lambda,x,y)$ continuous, using compactness of $I,$ we can find a  product  ball $ (B'=B(0,\epsilon))\times (B''= B(0,\epsilon)) \subset X\times Y $ such that  $ \| D_y k(\lambda, x,y)\| \leq \frac{1}{2}$  on $I\times B'\times B''$.  From this, we obtain that

\begin{align}\label{eq2}
\| k(\lambda,x,y) -k(\lambda,x,y')\| \leq \frac{1}{2} \|y-y'\|,
\end{align}
for any  $(\lambda,x) \in I\times B'$  and any $y,\, y' \in B''.$ Taking  a $\delta< \epsilon$ such that  $\|k(\la,x,0)\| = \|f_2(\lambda,x,0)\|\leq \frac{\epsilon}{2}$ on $B=B(0,\delta)$ and using \eqref{eq2} together with  $k(\la,0,0)=0,$  we obtain that $\| k(\lambda,x,y)\|< \epsilon,$ on  $I\times  B\times B''.$\\
Therefore, for any $(\lambda,x) \in I\times B$,  the map $k_{(\lambda,x)} (y):=  k(\lambda,x,y) $ is a strict contraction of $B''$ into itself. Defining $g(\lambda,x)$ to be the unique fixed point of  $k_{(\lambda,x)} \colon B'' \rightarrow B'' $,  the equation  \eqref{eq1} holds true. The continuity of $g$ follows from the fact that fixed points of continuous families of contractions depend continuously on the parameter.  By the implicit function theorem for $C^1$ maps, each $g_\lambda\equiv g(\lambda,\cdot) $ is differentiable. That the map $ (\lambda,x)\rightarrow Dg_\lambda(x) $ is continuous is a simple consequence of the formula for the derivative of an implicit function. Finally, a straightforward application of the chain rule gives \eqref{grad}, from which it also follows that $\bar \psi$  is a continuous family of $C^2$ functionals.
\end{proof}

Let us take as  $X,Y$ in Lemma \ref{lsreduction}  the pair  $H_n, H_n^\perp.$ Then we have a splitting   $f=(f^n_1, f^n_2),$  where 

\[f^n_1(\lambda, x,y) = P_n f(\lambda,x,y)\,   \hand\,  f^n_2(\lambda, x,y) =  (\Id -P_n)f(\lambda,x,y).\]
Moreover,  we have that $ D_y f^n_2 (0,0) = \tilde L^ n$ is an isomorphism   for  $n$  large  enough. Thus,  by the previous lemma we obtain a finite dimensional reduction  $\bar \psi^n \colon I\times B \rightarrow \R,$  $\bar f^n =\nabla_x \bar \psi^n \colon I \times B \rightarrow H_n $ for some ball $B \subset H_n$ centered at $0.$   

\begin{prop}\label{samemorse}
Let  $l^n_\lambda = D_x \bar f^n(\lambda, 0)$  be the Hessian of $\bar \psi^n_\lambda$  at $0. $  For $i =a,b $ and $n$ big enough, we have that  $0$ is a non-degenerate critical point of $\bar\psi_i^n$ and $$ \mu (\bar \psi^n_i, 0) :=\mu (l^n_i)= \mu (L^n_i).$$  
\end{prop}
 
\begin{proof} 
For $i=a,b$,  let $C^n_i = D g_i(0).$   From equation  \eqref{eq1} we obtain by implicit differentiation    
\begin{equation}\label{equ2}  
C^n_i  =  [D_y f^n_2(i,0,0)]^{-1}   D_x f^n_2(i,0,0)= (\tilde L^n_i)^{-1} (\Id -P_n) {L_i} _{| H_n}. 
\end{equation}
For $n$ big enough,   $\tilde L^n_i$  becomes as  close  as we wish to the isomorphism  $L_i$  which make the norms of its inverses uniformly bounded.  If we now use that $L_i=J+K_i$ and that $J$ commutes with $P_n$, we infer
\[C^n_i=(\tilde L^n_i)^{-1} (\Id -P_n)( {J+K_i}) _{| H_n}=(\tilde L^n_i)^{-1} (\Id -P_n) {K_i} _{| H_n}.\]
Now, since $(\Id -P_n) {K_i}$ converges to $0$ in norm, we obtain that $\|C^n_i\|\rightarrow 0$, $n\rightarrow\infty$. Observing that
 
\[l^n_i=P_n L_i(Id+C^n_i)= L^n_i + P_nL_iC^n_i,\]
we see that $ l^n_i -L^n_i \to 0$  in norm which implies that the Morse index of $l^n_i$ is defined and coincides with that of $L^n_i$ for sufficiently large $n$.
\end{proof} 

By construction,  the family  $\bar\psi$ and  the family $\psi$ have the same bifurcation points. Now, part $ i) $ of Theorem \ref{thm:bif}   follows from the propositions \ref{sflowfin}, \ref{samemorse} and:  

\begin{prop}\label{biffin}
Let  $U\subset \mathbb{R}^N $ be a  neighborhood of $0$ and let $\psi\colon I \times U \rightarrow\mathbb{R}$ be a  continuous family of  $C^2$  functionals such that $0$ is a critical point of  $\psi _\lambda \equiv \psi(\lambda,\cdot)$ for each $\lambda\in I.$ If the Hessians $l_a$ of $ \psi_a $ and  $l_b$  of $ \psi_b$ are invertible and $\mu(l_a)\neq\mu(l_b),$  then  the interval  $(a,b)$ contains at least one point  of bifurcation of critical points  of the family  $\psi_\lambda$  from the trivial branch.
\end{prop}

\begin{proof}
In  order to show that $\mu(l_a) \neq \mu(l_b)$ entails bifurcation, we will use  the continuation property of the Conley index of an isolated invariant set of a flow.  We will shortly recall below some basic facts about the Conley index. Here we follow essentially the presentation of the Conley index in \cite{Ba} and \cite{Ba1}).\\
A \emph {flow} on a locally compact metric space $X$ is a continuous map $\Phi \colon \mathbb{R} \times X \rightarrow X$ such that $\Phi(0,x) =x $ and $\Phi( t+s,x) = \Phi \left(t, \Phi(s,x)\right)$. As usual we will use subindices in order to denote the partial maps. We will denote by $ \gamma(x)  = \Phi_x(\mathbb{R})$ the orbit of a point $x$ and by $\alpha(x)$, $\omega(x)$ the alpha and omega limit respectively of the orbit passing through $x$. Given a subset $A$ of $X$ we will denote by  $I(A)$ the maximal invariant subset of $A$ , i.e.,  the union of all orbits  contained in $A$. A compact subset $N$ of $X$ is an \emph{ isolating neighborhood} if  any orbit $\gamma(x)$  contained in $N$ does not intersect $\partial N$. If $N$ is an isolating neighborhood,  the maximal invariant subset $I(N)$ of $N$ is a compact subset of  the interior of $N.$  An invariant subset $S$  of $X$ of  the form   $S= I(N)$ with $N$ an isolating neighborhood is called an \emph{isolated invariant set}.  An \emph{index pair} for an isolated invariant  set  $S$ is a pair $(N, E)$  of compact subsets of $X$ such that the closure of  $N\setminus E$ is an isolating neighborhood of $S$. The set $E\subset N,$ called the \emph{exit set}, is positively invariant  in $N$ and such that any orbit exiting $N$ in forward time must pass through $E.$ Every isolating neighborhood of the set  $S$ contains an index pair $(N,E)$ for $S$.\\
By definition, the Conley index $h(S)$ of an isolated invariant set $S$ is the homotopy type  of the pointed space $N/E$  obtained from $N$ by identifying all points in $E$ to a  base point.   That the (pointed) homotopy type of $N/E$  is independent of the choice of the index pair is one of the main results of Conley's theory. The invariance of the  Conley index  under conjugation of flows by homeomorphisms  allows to compute the index of an isolated hyperbolic equilibrium point $x_0$ of a complete  $C^1$- vector field $ f $  considered as an isolated invariant set of the associated flow. It turns out that $h(\{x_0\})$   is the homotopy type of a pointed sphere  $S^m$, where $m$  equals the dimension of the unstable manifold of  $x_0$  or, what is the same,  the number of characteristic exponents of $x_0$  with positive real part (see \cite{C} Chapter I , Section 4.3). In particular, if $x_0$ is a non-degenerate equilibrium point of a $C^2$ functional $\psi$, then  the Conley index of $x_0$ considered as an isolated critical  point of the flow associated to $  -\nabla \psi $  is the homotopy type of the sphere  $S^{\mu},$ where $\mu= \mu(\psi,x_0)$ is the Morse index of the critical point.  We will use a particular case of the continuation principle  for the Conley index in the form presented in \cite{Ba1} (cf. also \cite{Ba}).\\ 
A continuous map $\Phi \colon I \times \mathbb{R} \times X \rightarrow X $ such that each  $\Phi_\lambda$ is a flow induces a flow $\bar \Phi$ on $I\times X$  defined  by $\bar\Phi (t, \lambda,x) = (\lambda, \Phi(\lambda,t,x)).$ That  $\bar\Phi$ is a flow is clear from the definition. Moreover, the fibers of the projection $\pi \colon I\times X \rightarrow I $ are invariant  under  $\bar\Phi $,   which acts  on the  fiber   $ \{\lambda \} \times X $ as  $\Phi_\lambda.$ The continuation principle for the Conley index, in the formulation of  \cite[Theorem 3.3]{Ba1},    asserts  that if  $S$ is an isolated invariant set for $\bar \Phi$,  then  the section $S_\lambda =\{ x \in X | (\lambda,x) \in S\}$ is an isolated invariant set for $ \Phi_\lambda$ and  $h( S) = h(S_\lambda)$  for all $\lambda \in I.$ In particular, $h(S_\lambda)$ is independent of $\lambda$.\\
With this said we continue with the proof of the proposition. Assume that for $\epsilon$ small enough  there are no nontrivial vertical  critical points of the family $\psi$ on $\overline{B(0,\epsilon)}.$  Set $ f_\lambda(h) = - \nabla \psi_\lambda(h).$ Possibly after multiplication by a smooth bump function equal to one on  $\overline{B(0,\epsilon)}$,  we can assume without loss of generality that each vector field $f_\lambda$ is complete. Consider the homotopy $\Phi \colon I \times\mathbb{R} \times X \rightarrow X $  defined by the family of flows induced by the  differential equation  $ \dot h =   f(\lambda, h)$ and  take the flow $\bar \Phi$ constructed above.  Then the set $ S = I\times\{0\}$ is an isolated invariant set with isolating neighborhood $N = I \times \overline{B(0,\epsilon)}.$ Indeed, if there is an orbit of  $\bar\Phi$  passing through a point $(\lambda,h)$ with $h\neq 0$ and contained in $N$, then by the above discussion the orbit is of the form $\{\lambda \} \times \gamma(h)$ where  $\gamma (h)$ is the orbit of $\Phi_\lambda $ through $h$. Since  $\Phi_\lambda $ is of gradient type, the $\alpha $ and the $\omega$ limit of $\gamma(h)$ should be two different critical points of $\psi_\lambda$ in $\overline{B(0,\epsilon)},$ which is impossible.  But then,  by  the continuation principle  and the computation of the Conley index of a non-degenerate equilibrium point, we get

\[\mu(l_a) =\mu(\psi_a,0)= \mu(\psi_b,0) =\mu(l_b)\] 
contradicting the hypothesis.
\end{proof}

\vskip10pt

We now prove the second part of Theorem \ref{thm:bif}. Since  the singular set  $\Sigma(L)$ is finite, for each $\la_0  \in \Sigma(L) $ we can define {\it the spectral flow of $L$ across} $\la_0$  by

\[\sfl(L,\la_0)=\lim_{\varepsilon\rightarrow 0}\sfl(L,[\la_0-\varepsilon,\la_0+\varepsilon]).\]
That this limit exists is a consequence of the additivity and normalization properties of the spectral flow. By the first assertion of Theorem \ref{thm:bif}, $\la_0$ is a bifurcation point if $\sfl(L,\la_0)\neq 0.$

\begin{lemma}\label{sflred}
Let $L:I\rightarrow\Phi_S(H)$ be a path such that $L_\la$ fails to be invertible  only at $\la_0\in (a,b)$. Then $|\sfl(L,\la_0)|\leq\dim\ker L_{\la_0}.$
\end{lemma}

\begin{proof}
Since $L_{\la_0}$ is a self-adjoint Fredholm operator, there exists $\varepsilon>0$ such that $0$ is the only point in the spectrum of $L_{\la_0}$ in the interval $[-\varepsilon,\varepsilon]$.\\ 
Let us take  $I=[\la_0-\delta,\la_0+\delta],$  where  $\delta>0$ is  such that $\pm\varepsilon$ is not in the spectrum of $L_\la$ for all $\la\in I $. Let $P_\la$, $\la\in I$, denote the orthogonal projection onto the spectral subspace of $L_\la$ corresponding to the spectrum inside $[-\varepsilon,\varepsilon]$. Then $P:I\rightarrow\mathcal{L}(H)$ is continuous, for each $\la\in I,$ the projection $P_\la$ reduces $L_\la$ and the restriction of $L_\la$ to the kernel of $P_\la$ is an isomorphism. By possibly using a smaller $\delta>0$, we can assume that
$\|P_\la-P_{\la_0}\|<1,$
and by following the construction in \cite[Remark 4.4, Chap II]{Kato}, there is a path $U:I\rightarrow\mathcal{L}(H)$ of orthogonal isomorphisms such that

\[P_\la=U_\la P_{\la_0}U^\ast_\la,\quad \la\in I,\quad\text{and}\quad U_{\la_0}=\Id.\]
From the homotopy invariance property iv) of the spectral flow, it is readily seen as in the proof of the first part of Theorem \ref{thm:bif} that

\[\sfl(L,I)=\sfl(U^\ast LU,I).\]
Moreover, each $U^\ast_\la L_\la U_\la$ commutes with $P_{\la_0}$ and so is reduced by this projection. Setting $H_0=\ker L_{\la_0}$ and $H_1=H^\perp_0$, we obtain 

\[U^\ast_\la L_\la U_\la=(U^\ast_\la L_\la U_\la\mid_{H_0})\oplus (U^\ast_\la L_\la U_\la\mid_{H_1})\]
and we conclude from the property iii) of the spectral flow that

\[\sfl(U^\ast LU,I)=\sfl(U^\ast LU\mid_{H_0},I)+\sfl(U^\ast LU\mid_{H_1},I).\]
Since $L_\la$ is invertible for all $\la\neq \la_0$ and $U_{\la_0}=\Id$, we have  $\ker U^\ast_\la L_\la U_\la\subset H_0$ for all $\la\in I$ and consequently $U^\ast LU\mid_{H_1}$ is a path of invertible operators. Using i) and ii) we get:

\begin{align*}
&|\sfl(L,\lambda_0)|=|\sfl(U^\ast LU,I)|=|\sfl(U^\ast LU\mid_{H_0},I)|=\\
&=|\mu(U^\ast_{\la_0-\delta} L_{\la_0-\delta}U_{\la_0-\delta}\mid_{H_0})-\mu(U^\ast_{\la_0+\delta} L_{\la_0+\delta}U_{\la_0+\delta}\mid_{H_0})|\leq  \\ &\leq\dim H_0=\dim\ker L_{\la_0}.
\end{align*}
\end{proof}

Since  the set  $B(\psi)$ of  bifurcation points of $\psi$  is a subset of  $\Sigma(L),$ its cardinality  is a finite number $N.$ Using again the additivity and the normalization properties of the spectral flow, we obtain 

\begin{equation} \label{inequality}
|\sfl(L,I)|=\left|\sum_{\la \in \Sigma(L) }{\sfl(L,\la )}\right|\leq\sum_{\la \in \Sigma(L) }{|\sfl(L,\la)|}.
\end{equation}
On the other hand, since the points of $\Sigma(L)\setminus B(\psi)$ do not give any positive contribution to the right hand side of \eqref{inequality}, we have 

\[ |\sfl(L,I)|\leq\sum_{\la \in B(\psi) }{|\sfl(L,\la)|}  \leq Nm,\] by the previous lemma.
Consequently the number of bifurcation points $N$ is bounded from below by the quotient $|\sfl(L,I)|/m.$

%%%%%%%%%%%%%%%%%%%%%%%%%%%%%%%%%%%%%%%%%%%%%%%%%%%%%%%%%%%%%%%%%%%%%%%%%%%%%%%%%%%%%%%%%%%%%%%%%%%%%%%%%%%%%%%%%%%%%%%%%%%%%%%%%%%%%%%%%%%%%%%%%%%%%%%%%%%%%%%%%%%%%%%%%%%%%%%%%%%%%%%%%%%%%%%%%%%%%%%%%%%%%%%%%%%%%%%%%%%%%%%%%%%%%%%%%%%%%%%%%%%%%%%%%%%%%%%%%%%%%%%%%%%%%%%%%%%%%%%%%%%%%%%%%%%%%%%%%%%%%%%%%%%%%%%%%%%%%%%%%%%%%%%%%%%%%%%%%%%%%%%%%%%%%%%%%%%%%%%%%%

\section{Proof of Theorem \ref{thm:fam} and Corollary \ref{cor:fam} }\label{sec:fam}

In order to simplify notations, we will henceforth drop the domain  $I$ of the path from the notation and use $\sfl(L)$ to denote $\sfl(L,I).$\\
We will show at first the assertion i). Assume that $\gamma:I\rightarrow\Lambda$ is an admissible path such that $\sfl(L\circ\gamma)\neq 0$. Suppose  that $ \Lambda \setminus B(\psi)$ is connected. Join $\ga(1)$ with $\ga(0)$  by a path  $\tilde{\gamma}:I\rightarrow\Lambda$  such that $\tilde{\gamma}(I)\cap B(\psi)=\emptyset$. We must have $\sfl(L\circ\tilde{\gamma})=0.$ Otherwise, by Theorem \ref{thm:bif}, the  family of functionals $\phi \colon I \times U \ra \R$ defined by   $\phi(t, u) = \psi(\ga(t) ,u)$  would have a bifurcation point  $t_*.$ But then, $\ga(t_*)$ would be a point of intersection of $\tilde{\gamma}(I)$ with $B(\psi).$\\  
From the additivity of the spectral flow, we see that

\begin{align*}
\sfl(L\circ\gamma)=\sfl(L\circ\gamma)+\sfl(L\circ\tilde{\gamma})=\sfl(L\circ(\gamma\ast\tilde{\gamma})),
\end{align*}
where $\gamma\ast\tilde{\gamma}$ denotes the concatenation of $\gamma$ and $\tilde{\gamma}$. Since $\gamma\ast\tilde{\gamma}$ is a closed path, by assumption (A),  $\sfl(L\circ\gamma)=\sfl(L\circ(\gamma\ast\tilde{\gamma}))=0$, a contradiction.\\
Let us now prove ii). By choosing a point  $\la_0\in \Lambda\setminus \Sigma(L)$, we can associate an index $i(C)$ to each component of $\Lambda\setminus B(\psi)$ which contains elements of $\Lambda\setminus\Sigma(L)$ by defining  $i(C) = \sfl (L\circ \ga) $ where  $\ga$ is any admissible path joining $\la_0$ with a point of $C$. Much as in i),  $ \sfl (L\circ \ga) $ is  independent from the  choice of a point in $C\cap(\Lambda\setminus\Sigma(L))$,  and hence  the index is well defined. But then, for any admissible path $\ga$ we have that $$\sfl (L\circ \ga) = i(C_{\ga(0)}) - i(C_{\ga(1)}).$$ Therefore if the number of components were finite, the function  $ \ga \to \sfl (L\circ \ga)$ would  take only a finite number of values, contradicting the assumption.\\
In order to prove iii), we observe at first that,  given any admissible path $\ga$ such that  $L\circ\ga$  has only isolated singular points,   we can find another  path $\bar \ga$  with the same property  but which never returns to the same component.  For this, let $ \mu_1 < \mu_2 < \dots < \mu_k $ be the singular points of $L\circ \ga$ and set $\mu_0=0$, $\mu_{k+1}=1$.  We enumerate the components  of  $\Lambda \setminus B(\psi)$  traversed  by $\ga$  in an increasing order  according to the ordering  of $\mu_i.$ We obtain in this way  a list  of components $(C_0, C_1, \dots, C_k) $ with $\ga(0) \in C_0 $ and $\ga(1) \in C_k.$  If there is an index pair $i<j $ such that $C_i =C_j = C,$ we can obtain a shorter path $\ga'$ by connecting  $\ga(\mu_i)$ with $\ga(\mu_{j+1})$ by a path $\delta$ such that $\delta (0,1)\subset C.$ The path $\ga'$ has the same properties as $\ga$, but $C$ will appear in its list one time less. Iterating this procedure we obtain a path $\bar \gamma$ which never  traverses twice the same component.  Now the assertion follows from $ii)$ of Theorem \ref{thm:bif} applied to the path $L\circ \bar\ga.$\\
In order to prove Corollary \ref{cor:fam}  in the case when $\Lambda$  is a smooth manifold,  it is enough  to observe that, by density of transversality,  in the above reduction procedure  we can  approximate  the path $\delta $  with a differentiable path contained in the component $C$ that is  transversal to $\cS$ obtaining  in this way a new path  joining   $\ga(\mu_i)$ with $\ga(\mu_{j+1}),$  which composed with $L$ will still  have a finite number of singular points.  On topological manifolds, using the  dimension assumption, we can find inside $C$ a path $\delta$ which avoids $\Sigma(L) $ because sets of dimension $n-2$ or less cannot separate $C.$ 
\qed

%%%%%%%%%%%%%%%%%%%%%%%%%%%%%%%%%%%%%%%%%%%%%%%%%%%%%%%%%%%%%%%%%%%%%%%%%%%%%%%%%%%%%%%%%%%%%%%%%%%%%%%%%%%%%%%%%%%%%%%%%%%%%%%%%%%%%%%%%%%%%%%%%%%%%%%%%%%%%%%%%%%%%%%%%%%%%%%%%%%%%%%%%%%%%%%%%%%%%%%%%%%%%%%%%%%%%%%%%%%%%%%%%%%%%%%%%%%%%%%%%%%%%%%%%%%%%%%%%%%%%%%%%%%%%%%%%%%%%%%%%%%%%%%%%%%%%%%%%%%%%%%%%%%%%%%%%%%%%%%%%%%%%%%%%%%%%%%%%%%%%%%%%%%%%%%%%%%%%%%%%%

\section{Bifurcation of periodic orbits for  perturbations of autonomous Hamiltonian systems}\label{sec:apl}
We assume that $\Lambda$ is a connected topological space and $\mathcal{H}:\Lambda\times\mathbb{R}\times\mathbb{R}^{2n}\rightarrow\mathbb{R}$ is a continuous function such that each $\mathcal{H}_\lambda$ is $C^2$ and its first and second partial derivatives depend continuously on $\lambda\in \Lambda$. Moreover, we require that $\mathcal{H}(\lambda,t,u)$ is $2\pi$-periodic with respect to $t$, and that $\mathcal{H}(\lambda,t,0)=0$ for all $(\lambda,t)\in\Lambda\times\mathbb{R}$. Let us consider Hamiltonian systems

\begin{equation}\label{equation}
\left\{
\begin{aligned}
\sigma u'(t)+\nabla_u&\mathcal{H}(\lambda,t,u(t))=0,\quad t\in [0,2\pi]\\
u(0)&=u(2\pi),
\end{aligned}
\right.
\end{equation}  
where $\sigma$ denotes the standard symplectic matrix. Note that $u\equiv 0$ is a solution of \eqref{equation} for all $\lambda\in\Lambda$.  In what follows we shall also assume the two following conditions:

\begin{enumerate}
	\item[(H1)] There are constant $a,b\geq 0$ and $r>1$ such that
	 \begin{align*}
\begin{split}
|\nabla_u\mathcal{H}(\lambda,t,u)|&\leq a+b|u|^r,\\
|D_u\nabla_u\mathcal{H}(\lambda,t,u)|&\leq a+b|u|^r,\quad (\lambda,t,u)\in  \Lambda\times\mathbb{R}\times\mathbb{R}^{2n}.
\end{split}
\end{align*}
\item[(H2)] There exists a continuous family $A_\lambda$ of real time-independent symmetric $2n\times 2n$ matrices such that
\[\mathcal{H}(\lambda,t,u)=\frac{1}{2}\langle A_\lambda u,u\rangle+R(\lambda,t,u),\quad (\lambda,t,u)\in\Lambda\times\mathbb{R}\times\mathbb{R}^{2n},\]
where $\nabla_u R(\lambda,t,u)=o(\|u\|)$ as $\|u\|\rightarrow 0,$ uniformly in $(\la,t).$ 
\end{enumerate}
Let us recall that the Hilbert space $H^\frac{1}{2}(S^1,\mathbb{R}^{2n})$ consists of all functions $u:[0,2\pi]\rightarrow\mathbb{R}^{2n}$ such that
\begin{align*}
u(t)=c_0+\sum^\infty_{k=1}{a_k\sin\,kt+b_k\cos\,kt},
\end{align*}
where $c_0,a_k,b_k\in\mathbb{R}^{2n}$, $k\in\mathbb{N}$, and $\sum^\infty_{k=1}{k(|a_k|^2+|b_k|^2)}<\infty.$ The scalar product on $H^\frac{1}{2}(S^1,\mathbb{R}^{2n})$ is defined by

\begin{align*}
\langle u,v\rangle_{H^\frac{1}{2}}=\langle c_0,\tilde{c}_0\rangle+\sum^\infty_{k=1}{k(\langle a_k,\tilde{a}_k\rangle+\langle b_k,\tilde{b}_k)\rangle},
\end{align*}
where $\tilde{c}_0$ and $\tilde{a}_k,\tilde{b}_k$ denote the Fourier coefficients of $v\in H^\frac{1}{2}(S^1,\mathbb{R}^{2n})$.\\
Let

\[\Gamma:H^\frac{1}{2}(S^1,\mathbb{R}^{2n})\times H^\frac{1}{2}(S^1,\mathbb{R}^{2n})\rightarrow\mathbb{R}\]
be the unique continuous extension of the bounded bilinear form

\[\tilde{\Gamma}(u,v)=\int^{2\pi}_0{\langle\sigma u'(t),v(t)\rangle\,dt}\]
defined on the dense subspace $H^1(S^1,\mathbb{R}^{2n})$ of $H^\frac{1}{2}(S^1,\mathbb{R}^{2n})$, which consists of all absolutely continuous functions $u:S^1\rightarrow\mathbb{R}^{2n}$ having a square integrable derivative. We consider the map

\[\psi:\Lambda\times H^\frac{1}{2}(S^1,\mathbb{R}^{2n})\rightarrow\mathbb{R},\quad \psi_\lambda(u)=\frac{1}{2}\,\Gamma(u,u)+\int^{2\pi}_0{\mathcal{H}(\lambda,t,u(t))\,dt}.\] 
It is a standard result that each $\psi_\lambda$ is $C^2$ under the growth conditions (H1). Moreover, 

\[\langle \nabla_u\psi_\lambda  u,v\rangle_{H^\frac{1}{2}}=\Gamma(u,v)+\int^{2\pi}_0{\langle\nabla_u\mathcal{H}(\lambda,t,u(t)),v(t)\rangle\,dt},\quad v\in H^\frac{1}{2}(S^1,\mathbb{R}^{2n}),\]
and consequently, the critical points of $\psi_\lambda$ are precisely the weak solutions of the Hamiltonian system \eqref{equation}. In particular,$ \ u\equiv 0\in H^\frac{1}{2}(S^1,\mathbb{R}^{2n})$ is a critical point of  $\psi_\lambda,$ whose Hessian is defined by

\[\langle L_\lambda u,v\rangle_{H^\frac{1}{2}}=\Gamma(u,v)+\int^{2\pi}_0{\langle A_\lambda u(t),v(t)\rangle\,dt}.\] 
By the compactness of the embedding $H^\frac{1}{2}(S^1,\mathbb{R}^{2n})\hookrightarrow L^2(S^1,\mathbb{R}^{2n})$, the self-adjoint operators $L_\lambda$ are Fredholm. Moreover, $L_\lambda$ is invertible if and only if the matrix $A_\lambda$ is non-resonant, i.e., the spectrum of $\sigma A_\lambda$ does not contain integral multiples of the imaginary unit $i$.\\
We now define for any $2n\times 2n$ matrix $A$, a sequence of $4n\times 4n$ matrices by

\begin{align*}
L^0(A)=\begin{pmatrix}
A&0\\
0&A
\end{pmatrix},\quad L^k(A)=\begin{pmatrix}
\frac{1}{k}A&\sigma\\
-\sigma&\frac{1}{k}A
\end{pmatrix},\quad k\in\mathbb{N}.
\end{align*} 
Note that for $k$ sufficiently large, the matrices $L^k(A)$ become arbitrarily close to an invertible symmetric matrix of vanishing signature. Hence $\sgn L^k(A)=0$ for sufficiently large $k.$ This fact allows us to define  the  \textit{index}  $i(A)$ of the matrix $A$ by

\[i(A)=\frac{1}{2}\sum^\infty_{k=1}{\sgn L^k(A)}.\] 
It was shown in \cite[\S 1]{SFLPejsachowiczII} that, if  $\Lambda=[a,b]$ is a compact interval and the matrices $A_b$ and $A_a$ are non-resonant, so that $L$ has invertible ends, then 

\begin{equation}\label{indexdif}
\sfl(L)=i(A_b)-i(A_a).
\end{equation} 
As an immediate consequence of Theorem \ref{thm:fam} we obtain:

\begin{prop}\label{prop:Ham}
Let $\Lambda$ be a connected  topological space. Assume that {\rm (H1)} and {\rm (H2)}  hold for $\mathcal{H}$.
\begin{itemize}
	\item[i)] If there exist  $\lambda,\mu \in\Lambda$ such that the matrices $A_{\lambda}, A_{\mu}$ are non-resonant and $i(A_{\lambda})\neq i(A_{\mu})$, then $ \Lambda\setminus B(\psi)$ is disconnected.\\
	
		\item[ii)] If there exists a sequence $\{\lambda_n\}_{n\in\mathbb{N}}\subset\Lambda$ such that $A_{\lambda_n}$ is non-resonant for all $n\in\mathbb{N}$ and 	
	$\lim_{n\rightarrow\infty}|i(A_{\lambda_n})|=\infty,$
	then $ \Lambda\setminus B(\psi)$ has infinitely many path components. \\
	
	 \item[iii)]  If  $B(\psi)=  \{\la : A_\la \text{ is resonant}\},$   then  any path $\ga$  joining  two non-resonant parameters $\lambda$ and $\mu$ such that  $A\circ\ga$  has only isolated resonant points must  cross at least   $\frac{|i(A_{\lambda})- i(A_{\mu})|}{2n} +1 $  components of $\Lambda \setminus B(\psi).$ 
\end{itemize}
\end{prop}

From Corollary \ref{cor:fam}  we obtain:

\begin{cor}\label{cor:new}
Let $\Lambda, \mathcal{H}$ be as in the previous proposition. Assume moreover  that $\Lambda$ is a smooth $n$-manifold and that  the map $ A $ is smooth and transversal  to the variety  $\cR$  of all resonant matrices.  Assume moreover  that  there exists  a path $\ga\colon (-\infty ,\infty)\ra \Lambda$ such that $A\circ\ga$ is  a  positive path of symmetric  matrices.   Then,  irrelevant of the choice of the higher order perturbation $R,$  the number of connected components of $\Lambda \setminus B(\psi)$ is not less than $ \frac{k}{2n} +1,$  where $k$ is the number of intersections of the spectrum of $\sigma A\circ \ga$ with $i\Z.$  
\end{cor}
 
For  example:  let  $\cA$  be the space of  all symmetric $2n\times 2n$ matrices and consider the  problem of bifurcation of periodic solutions for  the parametrized system of Hamiltonian equations \eqref{equation}   with  Hamiltonian  function $\mathcal{H}(A,t,u)=\langle Au,u\rangle+R(A,t,u),$ where  $A\in \cA$ and $R$ is  any family of $C^2$  periodic functions continuously parametrized by $\cA$, such that $\nabla_u R(A,t,u)=o(\|u\|)$ uniformly in $(A,t).$  The positive path $\ga(\la) = \la\Id,$  is resonant for all $\la \in \Z.$ Therefore,  the set $\cA \setminus B(\psi)$ must  have infinitely many path components.  Note however  that  $\Sigma(L) = B(\psi)=\cR$ in this case.

%%%%%%%%%%%%%%%%%%%%%%%%%%%%%%%%%%%%%%%%%%%%%%%%%%%%%%%%%%%%%%%%%%%%%%%%%%%%%%%%%%%%%%%%%%%%%%%%%%%%%%%%%%%%%%%%%%%%%%%%%%%%%%%%%%%%%%%%%%%%%%%%%%%%%%%%%%%%%%%%%%%%%%%%%%%%%%%%%%%%%%%%%%%%%%%%%%%%%%%%%%%%%%%%%%%%%%%%%%%%%%%%%%%%%%%%%%%%%%%%%%%%%%%%%%%%%%%%%%%%%%%%%%%%%%%%%%%%%%%%%%%%%%%%%%%%%%%%%%%%%%%%%%%%%%%%%%%%%%%%%%%%%%%%%%%%%%%%%%%%%%%%%%%%%%%%%%%%%%%%%%

\section{Extended spectral flow and comparison theorems}\label{sec:comp}  
 
 For what follows it will be convenient for us to extend the definition of the spectral flow to general paths  in $\fs$ with not necessarily invertible end-points.  For a non-invertible operator $T\in \fs$, $0$  is always  an isolated eigenvalue of finite multiplicity.  Hence there exists a $\delta>0$ such that  $T+ \la\,\Id$ is invertible for
$0<\la\leq \delta.$  We define the spectral  flow of a general path $L\colon I \ra \fs$ as the spectral flow of the translated path 

\begin{equation}\label{eq:specgen}
    \spf(L) \, = \, \spf (L + \delta\,\Id),
\end{equation} 
where  $\delta>0$ is as above. Clearly, the right hand side  does not depend on the choice of $\delta$. The resulting function is  additive under concatenation and direct sum. It is homotopy invariant under homotopies keeping the end-points fixed. As a matter of fact, it is invariant under homotopies which keep the end points in a fixed stratum of $\Sigma,$ but we will not make use of this property here.\\
The cone of positive operators defines an order in $\mathcal L (H)$  compatible with its topology.  Given two bounded self-adjoint operators $T,S\in\mathcal{L}_S(H)$ we say that $T \geq S$ whenever $T-S$ is  positive, and that  $T > S$  whenever $T-S$ is positive definite, i.e.,  if there exists $c>0$ such that $\ip{(T-S)h}{h} \geq c\, \|h\|$  for all $h\in H.$   A path $L \colon I\ra  \cL_s(H)$ of self-adjoint operators is {\it non-decreasing}  if $L_\la \geq L_\mu$ whenever $\la \geq \mu$ and it is  {\it strictly increasing} if $L_\la > L_\mu$ whenever $\la>\mu.$ 

\begin{theorem}\label{thm:comparison} Let $H\colon [0,1] \times [a,b] \ra  \fs$ be a homotopy such that $H(\cdot,a) $ is non-increasing  and $H(\cdot,b)$ is non-decreasing,  then  
\begin{equation*}    
\spf H(0,\cdot) \leq \spf H(1,\cdot).
\end{equation*}
\end{theorem} 

If two paths $L$ and  $M$ are such that $L-M$ is a path of compact operators,  then  the family  $H(t,\la)= M + t(L-M)$ is a homotopy in $\fs$  between them.  It verifies the hypothesis of the previous theorem whenever $L_a\leq M_a$  and $M_b\leq L_b.$  Therefore we have: 

\begin{cor} \label{cor:comparison} 
Let $L, M \colon [a,b] \ra \fs$ be  paths such that  $L_\la -M_\la $ is  compact for each $\la \in I.$  If $L_a\leq M_a$  and $M_b\leq L_b,$ then 
\begin{equation*}  \label{eq: geqcomp}
\sfl M \leq \sfl L.
\end{equation*}
\end{cor}
 
Theorem \ref{thm:comparison} is an easy consequence of the following proposition:
  
\begin{prop} \label{posit}
If $L\colon [a,b] \ra \fs$ is non-decreasing, then $\spf(L)\geq0.$
\end{prop}

\begin{proof}
Let us choose, for each $\la \in I=[a,b],$ a ball $ B_\la =B( L_\la, \ve_\la )\subset \fs$  and  take  a partition $\pi$ of  $I$ of mesh less than the Lebesgue number of the covering $\cU =\{ L^{-1}(B_\la) : \la \in I\}.$  Then the  image  $L(I') $ of any interval  $I'$ of the partition is contained in some  ball $B_{I'}\subset \fs.$\\
Select an interval  $J=[c,d]$ of the partition $\pi,$ and the  straight line segment  $$N_\la =(d-c)^{-1}[ (\la-c)L_{d  }+ (d-\la) L_{c }] $$  joining $L_{c}$ with $L_{d}.$ Consider the affine  homotopy  $H(t,\la ) = (1- t)\,M_\la  + t\,N_\la $  between  the path $M =L_{|J}$  and  $N.$  Each $H(t,\la)$ is an element of $\fs$  because $ L_{|J}(J)$  is contained in $ B_J.$  Since $H$  is a homotopy  in $\fs$ with fixed end-points,  we have  $\spf M = \spf N = \spf (N + \delta\, \Id),$  where $\delta$ is as in \eqref{eq:specgen} and small enough in order to have  $(N + \delta\, \Id)(J) \subset B_J.$\\
Now, if we take a sufficiently small $ \mu>0$, we have that the segment $R,$ given by  $$ R_\la = (d-c)^{-1}[ (\la-c)(L_{d}+ \mu\, \Id) + (d-\la) (L_{c} + \mu\,\Id)], $$ is homotopic in $\fs$ to $N$  by an affine homotopy keeping end-points invertible.  Therefore,  $\spf  M = \spf R .$  But now $R$ is a differentiable  path and by direct calculation we see that, whenever $\ker R_\la\neq 0$, the crossing form  $\mathcal{Q}( R, \la)$ is positive definite and hence non-degenerate. The spectral flow  in this case  is simply the sum of the dimensions  of  the corresponding kernels. Hence  

\[\spf(L_{|J})=\spf(R)=\sum_ {\la \in J} \dim \ker R_\la \geq 0.\]
The proposition is obtained by summing over all intervals of the chosen partition.
\end{proof}

In order to  conclude the proof of the comparison theorem  from the above  proposition,  it is enough to recall that the spectral flow along a closed path is invariant by free homotopies. For a proof of this fact see \cite[Proposition 3.7]{FPR}. In particular, since the  rectangle $D=[0,1]\times [a,b]$ is contractible, the spectral flow  of the restriction of $H$ to its  boundary $\partial D$ must be zero. From this it follows that $\spf (H(0,\cdot)) \leq \spf (H(1,\cdot)), $ since  $\spf (H(1-\cdot,a))= -\sfl(H(\cdot,a))\geq 0 $  and $\spf( H(\cdot,b)) \geq 0$ by Proposition \ref{posit}.

%%%%%%%%%%%%%%%%%%%%%%%%%%%%%%%%%%%%%%%%%%%%%%%%%%%%%%%%%%%%%%%%%%%%%%%%%%%%%%%%%%%%%%%%%%%%%%%%%%%%%%%%%%%%%%%%%%%%%%%%%%%%%%%%%%%%%%%%%%%%%%%%%%%%%%%%%%%%%%%%%%%%%%%%%%%%%%%%%%%%%%%%%%%%%%%%%%%%%%%%%%%%%%%%%%%%%%%%%%%%%%%%%%%%%%%%%%%%%%%%%%%%%%%%%%%%%%%%%%%%%%%%%%%%%%%%%%%%%%%%%%%%%%%%%%%%%%%%%%%%%%%%%%%%%%%%%%%%%%%%%%%%%%%%%%%%%%%%%%%%%%%%%%%%%%%%%%%%%%%%%%

\section{ Estimates for the number of bifurcation points for general systems} \label{sec:est}
In this section we estimate from below  the number of bifurcation points for periodic orbits bifurcating from a given branch of a family of non-autonomous  Hamiltonian systems. Our aim  is to obtain  the estimates  directly  in terms of the coefficients of the linearization along the trivial branch.\\
Here  the  parameter space  $\Lambda$ will be the interval  $[0,1].$  Moreover we assume that  $\mathcal{H}: [0,1]\times\mathbb{R}\times\mathbb{R}^{2n}\rightarrow\mathbb{R}$ is a continuous  function, which is  $2\pi$-periodic with respect to $t,$ and such that each $\mathcal{H}_\lambda$ is $C^2$ and its first and second partial derivatives depend continuously on $(\lambda,t,u)$.  We also assume  $\mathcal{H}(\lambda,t,0)=0$ for all $(\lambda,t)\in I\times\mathbb{R}$ and take as the trivial branch of periodic solutions  the stationary branch $u(t)\equiv 0.$   It is shown in \cite{FPR} that bifurcation from a general branch of periodic solutions can be easily reduced to the above case.\\
Much as before, we are looking for  points  of bifurcation of nontrivial  $2\pi$-periodic solutions of \eqref{equation} from the stationary branch,  but here we do not assume that the system \eqref{equation}  is a higher order perturbation of an autonomous system. Namely,  while keeping  the assumption (H1) we  drop (H2).
We still have that critical points of

\[ \psi_\lambda(u)=\frac{1}{2}\,\Gamma(u,u)+\int^{2\pi}_0{\mathcal{H}(\lambda,t,u(t))\,dt},\]  
are weak solutions of \eqref{equation}.  Moreover the Hessian $L_\la$  of $\psi_\la $ at $u\equiv 0$  is given  by 
\begin{equation} \label{Hessian} \langle L_\lambda u,v\rangle_{H^\frac{1}{2}}=\Gamma(u,v)+\int^{2\pi}_0{\langle A_\lambda(t) u(t),v(t)\rangle\,dt},  \end{equation} where  $A_\la(t)= D_u\nabla_u\mathcal{H}_\la (t,0).$\\ 
In the general case of a time depending system $\sigma u'(t)+ A_\la(t) u(t) $ we do not have an index defined directly in terms of the coefficients as before. However, the spectral flow $\sfl(L,I)$ can still  be computed from  the relative Conley-Zehnder index of the path  $\{P_\la\}_{\la\in I} $ of Poincar\'{e} maps (cf. \cite{SFLPejsachowiczII}). On its turn, the relative Conley-Zehnder index of the path  $P$ can be explicitly computed from  the eigenvalues of $P_0$ and $P_1.$  Nevertheless  the Poincar\'{e} map  can be only obtained by integrating the linearization  and  cannot be considered as  computable from the coefficients.\\
When $A_i (t) \equiv A_i,  i=0,1,$ are constant, then one can  use  \eqref{indexdif} in order to estimate from below by  $\frac{|i(A_1)-i(A_0)|}{2n}$  the number of bifurcation points  for paths verifying  the hypothesis of  ii)  in Theorem \ref{thm:bif}.  Using the  Floquet reduction one  can always deform the linearized equation to one with constant coefficients.  However,  this procedure involves the  monodromy  operator too. In view of this, the estimate provided below, while rough, appears to be reasonable.  It is based on the numerical range of the matrices $A_i(t), i=0,1.$\\ 
Let  $\{\mu^i_1(t)\leq \mu^i_2(t) \leq  \dots \leq \mu^i_{2n}(t)\}$ be the eigenvalues   of $A_i(t), i=0,1.$  Set 
$ \alpha_i = \inf \{\mu^i_1(t): t \in I \}$ and $\beta_i = \sup \{\mu^i_{2n}(t): t\in I\}.$  Then we have 

\begin{equation}\label{bounds}  
\alpha_i\, \Id \leq A_i (t) \leq \beta_i \,\Id,\quad i=0,1.
\end{equation} 
Consider the path $B_\la =( \beta_0 + \la (\alpha_1 - \beta_0))\,\Id $ and let $M$ be the path of operators  on ${H^\frac{1}{2}}$  defined by 

\begin{equation}\label{weaksol}
\langle M_\lambda u,v\rangle_{H^\frac{1}{2}}=\Gamma(u,v)+\int^{2\pi}_0{\langle  B_\la u(t),v(t)\rangle\,dt}.
\end{equation} 
Then $L$ and $M$ are Calkin equivalent, since both are compact perturbations of the operator $\cD$ representing the quadratic form $\Gamma$  against the scalar product in ${H^\frac{1}{2}}.$ Moreover, 

\[\langle( L_\lambda - M_\la ) u,v\rangle_{H^\frac{1}{2}}=\int^{2\pi}_0{\langle (A_\lambda(t)-B_\la)  u(t),v(t)\rangle\,dt}.\] 
Taking $\la =0,1$ and using \eqref{bounds} we see that the paths $L$ and $M$ verify the hypothesis of Corollary  \ref{cor:comparison}. Therefore,

\begin{equation}\label{sflM}
\sfl M\leq \sfl L.
\end{equation}
In the same way  taking  $C_\la = \alpha_0 + \la (\beta_1 - \alpha_0)\, \Id$ and  applying  Corollary \ref{cor:comparison} to the  operator path $N$ induced on ${H^\frac{1}{2}}$  by   $ \sigma u'(t)+C_\la u(t),$ we obtain 

\begin{equation} \label{sflN}
\sfl L\leq \sfl N. 
\end{equation}
On the other hand, it follows from the definition of the extended spectral flow in \eqref {eq:specgen} that the formula \eqref{murel}, relating the spectral flow of a path of compact perturbations of a fixed operator  with the relative Morse  index of its end-points holds even when the end-points are non-invertible, provided that  we consider as $ E^+(T)$ the spectral subspace corresponding  the spectrum of $T$  in $(0, +\infty) $ and $E^-(T)$  the spectral subspace of the spectrum in $(-\infty,0].$ With the above setting we have:

\begin{equation}\label{murel2}
\begin{aligned}\sfl(M)&=\mu_{rel}(M_0,M_1)\\&=
\dim( E^-(M_0)\cap E^+(M_1) ) - \dim( E^-(M_1)\cap E^+(M_0) ),
\end{aligned}\end{equation}
an similarly for $N.$\\
Let us introduce  two  unbounded self-adjoint operators $\cM_0, \cM_1$  with domain $$\cD =H^1(S^1; \R^{2n})\subset L^2(S^1; \R^{2n}).$$ Namely, $\cM_0 = \sigma u'(t) + \beta_0 u(t) $ and  $\cM_1= \sigma u'(t) + \alpha_1 u(t).$   Since   $H^\frac{1}{2}(S^1,\mathbb{R}^{2n})$ is the form space of the operator $\sigma u'(t),$  it follows that  $M_i$ are the form extensions of $\cM_i.$  While the spectra of  $M_i$ and $\cM_i$ are necessarily different, it follows from the regularity  of solutions of \eqref{weaksol} that  they have the  same "positive" and "negative"  eigenvectors. Hence we can compute the right hand side of \eqref{murel2}, in terms of  $E^\pm(\cM_i)$ as well. For this, let us identify $\R^{2n}$ with $\C^n\equiv \R^n + i\R^n$ via   \[u=(x,y)\equiv z= x+iy.\] Under the above identification, the action of $\sigma$  coincides with the multiplication by $i$  and $\cM_i, i=0,1, $ become:   

\begin{equation}\label{B1}
\begin{array}{ccccc} \begin{cases}iz^\prime(t)+ \beta_0 z(t)=0 
\\z(0) = z(2\pi)\end{cases} & \quad \hand \quad &   \begin{cases}iz^\prime(t)+\alpha_1 z(t)=0 
\\z(0) = z(2\pi).\end{cases} & & \end{array}
\end{equation}
Their eigenvalues are of the form $ m + \beta_0$ and  $ m + \alpha_1,$ with $m \in \Z$ respectively, with (complex) multiplicity $n.$\\
Given real numbers $\mu$ and $\nu,$ define 
\begin{equation}\label{Delta} \Delta (\mu,\nu) = \left\{
\begin{aligned}
\ \  \ &\#\{ i\in\Z: \mu \leq i < \nu \}  &\hif \   \mu\leq\nu\\
-&\#\{ i\in\Z : \nu\leq i < \mu \}  &\ \ \hif \ \nu\leq\mu.
\end{aligned} \right.
\end{equation}
Then clearly 

\[\mu_{rel}(M_0,M_1)=\mu_{rel}(\cM_0,\cM_1)=  2n\, \Delta (\beta_0, \alpha_1).\]
A similar computation shows that 

\[\mu_{rel}(N_0,N_1)= \mu_{rel}(\cN_0,\cN_1)=  2n\, \Delta(\alpha_0, \beta_1).\]
By \eqref{sflM}, \eqref{sflN}  and \eqref{murel2} we have 

\[2n\, \Delta (\beta_0, \alpha_1)=\mu_{rel}(M_0,M_1) \leq \sfl L \leq \mu_{rel}(N_0,N_1)= 2n\, \Delta(\alpha_0, \beta_1).\]    
From this and Theorem \ref{thm:bif} ii) we can conclude that   
 
\begin{prop} \label{prop:last}  
Assume that the linearization of the problem  \eqref{equation}   along the stationary branch 

\begin{equation}\label{equation1}
\left\{
\begin{aligned}
\sigma u'(t)&+A_\la(t)u(t)=0 \\
u(0)&=u(2\pi),
\end{aligned}
\right.
\end{equation}    
admits  only trivial solutions for all but a finite number of values of $\la.$\\
Then

\begin{itemize} 
\item[i)]  If $\beta_0 < \alpha_1$, the family \eqref{equation}  has at least $\Delta (\beta_0, \alpha_1)$ points of bifurcation of periodic solutions from the stationary branch.
\item[ii)]  If $ \beta_1 < \alpha_0$,   the family \eqref{equation}  has at least  $-\Delta( \alpha_0, \beta_1)$  bifurcation points. 
\end{itemize}
\end{prop} 

\begin{rem}
{\rm  Notice that we do not have to assume that the path is admissible. Indeed, it is enough to apply Theorem  \ref{thm:bif} ii) on $[ \delta, 1-\delta]$ for $\delta$ small enough.} 
\end{rem}

%%%%%%%%%%%%%%%%%%%%%%%%%%%%%%%%%%%%%%%%%%%%%%%%%%%%%%%%%%%%%%%%%%%%%%%%%%%%%%%%%%%%%%%%%%%%%%%%%%%%%%%%%%%%%%%%%%%%%%%%%%%%%%%%%%%%%%%%%%%%%%%%%%%%%%%%%%%%%%%%%%%%%%%%%%%%%%%%%%%%%%%%%%%%%%%%%%%%%%%%%%%%%%%%%%%%%%%%%%%%%%%%%%%%%%%%%%%%%%%%%%%%%%%%%%%%%%%%%%%%%%%%%%%%%%%%%%%%%%%%%%%%%%%%%%%%%%%%%%%%%%%%%%%%%%%%%%%%%%%%%%%%%%%%%%%%%%%%%%%%%%%%%%%%%%%%%%%%%%%%%% 

\vspace{1cm}
Jacobo Pejsachowicz\\
Dipartimento di Scienze Matematiche\\
Politecnico di Torino\\
Corso Duca degli Abruzzi, 24\\
10129 Torino\\
Italy\\
E-mail: jacobo.pejsachowicz@polito.it

\vspace{1cm}
Nils Waterstraat\\
Dipartimento di Scienze Matematiche\\
Politecnico di Torino\\
Corso Duca degli Abruzzi, 24\\
10129 Torino\\
Italy\\
E-mail: waterstraat@daad-alumni.de

\end{document}